\documentclass{article}

\usepackage[all,ps,arc]{xy}
\usepackage{amsmath,amssymb,latexsym}
\usepackage{amsfonts}
\usepackage{amsthm}
\usepackage{color}
\usepackage{appendix}

\newcommand{\cC}{\ensuremath{\mathcal{C}}}
\newcommand{\cE}{\ensuremath{\mathcal{E}}}
\newcommand{\Cat}{\ensuremath{\mathsf{Cat}}}
\newcommand{\Ext}{\ensuremath{\mathsf{Ext}}}
\newcommand{\Gpd}{\ensuremath{\mathsf{Gpd}}}
\newcommand{\Fib}{\ensuremath{\mathsf{Fib}}}
\newcommand{\KFib}{\ensuremath{\mathsf{Fib}_\mathcal{C}}}
\newcommand{\OpFib}{\ensuremath{\mathsf{OpFib}}}
\newcommand{\PsFib}{\ensuremath{\mathsf{PsFib}}}
\newcommand{\ps}{\ensuremath{/\!/}}

\newcommand{\op}{\ensuremath{\mathrm{op}}}

\usepackage[colorlinks=true]{hyperref}
\hypersetup{allcolors=[rgb]{0.1,0.1,0.4}}

\author{Alan S.~Cigoli, Sandra Mantovani, and Giuseppe Metere}

\title{Discrete and conservative reflections of fibrations}

\newtheorem{Theorem}{Theorem}[section]
\newtheorem{Lemma}[Theorem]{Lemma}
\newtheorem{Definition}[Theorem]{Definition}
\newtheorem{Proposition}[Theorem]{Proposition}
\newtheorem{Corollary}[Theorem]{Corollary}

\newtheorem{Remark}[Theorem]{Remark}
\newtheorem{Example}[Theorem]{Example}

\begin{document}

\maketitle

\begin{abstract}
We focus on two factorization systems for opfibrations in the 2-category $\Fib(B)$ of fibrations over a fixed base category $B$. The first one is the internal version of the so called \emph{comprehensive factorization}, where the right orthogonal class is given by internal discrete opfibrations. The second one has as its right orthogonal class internal opfibrations in groupoids, i.e.~with groupoidal fibres.
These factorizations can be obtained by means of a single step 2-colimit. Namely, their left orthogonal parts are nothing but suitable coidentifiers and coinverters respectively. We will show how these results follow from their analogues in \Cat. To this end, we first provide suitable conditions on a 2-category \cC, allowing the transfer of the construction of coinverters and coidentifiers from \cC~to $\Fib_{\cC}(B)$.
\end{abstract}



\section{Introduction}

The starting point of the work \cite{CMMV17} was to study from a fibrational point of view the notion of \emph{regular span}, introduced by Yoneda in \cite{Y60} as a formal categorical setting in order to reformulate the classical theory of $\Ext^n$ functors. The results in \cite{CMMV17} reveal that a regular span $S$
is nothing but a cartesian functor with codomain a product projection:
\[
\xymatrix@!=4ex{
X \ar[dr]_{S_2} \ar[rr]^S & & A \times B \ar[dl]^{pr_2} \\
& B
}
\]
which is in addition an internal opfibration in $\Cat/B$ (i.e.~a \emph{fibrewise opfibration}---see Definition \ref{def:fw}). This enlights also the difference between regular spans and two-sided fibrations, which in turn were characterized in \cite{BournPenon} as internal opfibrations, with codomain a product projection, in the 2-category $\Fib(B)$ of fibrations over $B$.

This fibrational interpretation makes it possible to reformulate Yoneda's Classification Theorem of \cite{Y60} as the result of a canonical factorization, yielding a reflection of regular spans into profunctors, i.e.\ two-sided \emph{discrete} fibrations. Actually, in \cite{CMMV17} it is shown that such a factorization exists not only for regular spans, but for any fibrewise opfibration $p$ with codomain a split fibration. Moreover, it turns out that this factorization, performed via a coidentifier (Lemma 3.12 in [\emph{loc.~cit.}]), is the internal version in $\Fib(B)$ of the \emph{comprehensive factorization} introduced in \cite{StW} for \Cat. In other words, $p$ is the composite of an initial morphism in $\Fib(B)$ with an internal discrete opfibration, which is the same as a fibrewise opfibration whose fibres are discrete (Corollary 2.9 in \cite{CMMV17}).

It is natural to ask if, for fibrewise opfibrations, replacing the coidentifier with a coinverter in the construction above, we get as a comparison a fibrewise opfibration whose fibres are groupoids. To this end, we need first to detect some sufficient conditions to transfer the construction of coidentifiers and coinverters from a 2-category \cC\ to the 2-category $\KFib(B)$ of internal fibrations over a fixed object $B$ (see Proposition \ref{prop:coinv} and Proposition \ref{prop:coid}). This happens when the 2-monad $R\colon\cC/B\to\cC/B$, whose pseudo-algebras define internal fibrations (in the sense of Street \cite{Street74}), preserves coidentifiers and coinverters of identees. Under this assumption $(\dagger)$, the coinverter (coidentifier) $q$ of the identee $\kappa$ of a fibration $f$ 
\begin{equation} \label{fact}
\xymatrix@!=5ex{
K \ar@/^2ex/[r]^{d_0}_{}="1" \ar@/_2ex/[r]_{d_1}^{}="2" \ar@{=>}"1";"2"^{\kappa} & A \ar[d]_{f} \ar[r]^{q} & Q \ar[dl]^{s} \\
& B
}
\end{equation}
induces a  comparison morphism $s$, which is still a fibration in \cC\  (see Corollary \ref{cor:fib.to.frac} and Corollary \ref{cor:fib.to.sfrac}). 

It is important to point out that, for such an $s$, as for any other isofibration, having groupoidal fibres is the same as being conservative, as we show in Corollary \ref{cor:isofib.cons=gpd(isofib)} in the general context of a finitely complete 2-category. 

Now we have two facts. First, we prove that the above described behaviour of coinverters (coidentifiers) for  fibrations (and opfibrations) holds in \Cat\ and in $\Fib(B)$, for any  $B$, precisely because in these cases, condition $(\dagger)$ and its dual $(\dagger')$ are fulfilled (Corollary \ref{cor:cat.dag} and Proposition \ref{prop:dag.to.fib}).

On the other hand,  it is known that any functor can be factorized through a conservative functor, thanks to a factorization system for \Cat\ which is obtained by a (possibly transfinite) iteration of the invertee/coinverter construction. Actually, we show in Corollary \ref{cor:isofib.cons=gpd(isofib)} that for any isofibration, the coinverter of the invertee is the same as the coinverter of the identee, so that our factorization $f=sq$ in diagram (\ref{fact}) realizes the first step of the above mentioned  construction. Moreover, this first step is sufficient to produce the desired factorization for each fibration (respectively opfibration) $f$ in \Cat: considering the construction (\ref{fact}) for such an $f$, $s$ turns out to be a fibration (resp. opfibration)  in groupoids, i.e.\ conservative (Proposition \ref{prop:isocompr.cat}).

The same phenomenon occurs for the factorization system in \Cat\ given by (sequence of coidentifiers, discrete functor), which, when restricted to fibrations (opfibrations), reduces to a single application of the identee/coidentifier construction and it realizes the comprehensive factorization. This result is proved in Proposition  \ref{prop:comprehensive.cat} in the case of a 2-category $\Cat(\cE)$ of internal categories where the construction of the comprehensive factorization of any functor provided in \cite{StV} is still valid, as, for example, when \cE\ is a finitely cocomplete locally cartesian closed category.

In the last part of the paper, we extend the above results to  fibrewise opfibrations in $\Fib(B)$, relying on the pseudo-functorial interpretation  of opfibrations in \Cat. 
This way,  we prove in Proposition \ref{prop:refl.gpd.fib} that every fibrewise (resp.~internal) opfibration $p$ in $\Fib(B)$ admits a factorization
\[
\xymatrix@!=5ex{
A \ar@/^3ex/[rr]^{p} \ar[dr]_{f} \ar[r]_-{q'} & Q' \ar[d]_(.4){h'} \ar[r]_-{s'} & C \ar[dl]^{g} \\
& B
}
\]
where $q'$ is the coinverter of the identee of $p$ and $s'$ is a fibrewise (resp.~internal) opfibration in groupoids in $\Fib(B)$. Such a factorization coincides with the one given by (sequence of coinverters, conservative functor) in \Cat.

The analogous results hold when replacing coinverters with coidentifiers and opfibrations in groupoids with discrete opfibrations (Proposition \ref{prop:comprehensive.int}).

\section{Review of internal fibrations}

Let \cC\ be a finitely complete 2-category \cite{Street76}. For a fixed object $B$ in \cC, we shall denote by $\cC/B$ the comma 2-category over $B$ and by $\cC\ps B$ the pseudo-comma 2-category over $B$.

We shall denote as follows the (strict) comma objects in \cC\ of identities, along identities on the left and on the right respectively, and iso-comma along identities:
\begin{equation} \label{diag:commas}
\begin{aligned}
\xymatrix@!=4ex{
B/B \ar[d]_{d_0} \ar[r]^-{d_1} & B \ar[d]^{1} & B/f \ar[r]^-{d_1} \ar[d]_{Rf} & A \ar[d]^f & f/B \ar[r]^-{Lf} \ar[d]_{d_0} & B \ar[d]^{1} & f/_{\cong} B \ar[r]^-{If} \ar[d]_{w_f} & B \ar[d]^1 \\
B \ar[r]_{1} \ar@{}[ur]|(.3){}="1" \ar@{}[ur]|(.7){}="2" \ar@{=>}"1";"2"_{\mu_B} & B & B \ar[r]_{1} \ar@{}[ur]|(.3){}="3" \ar@{}[ur]|(.7){}="4" \ar@{=>}"3";"4"_{\varphi_f} & B & A \ar[r]_{f} \ar@{}[ur]|(.3){}="5" \ar@{}[ur]|(.7){}="6" \ar@{=>}"5";"6"_{\psi_f} & B & A \ar[r]_{f} \ar@{}[ur]|(.3){}="7" \ar@{}[ur]|(.7){}="8" \ar@{=>}"7";"8"_{\omega_f}^{\sim} & B
}
\end{aligned}
\end{equation}
One can extend the assignment $f\mapsto Lf$ to 1-cells and 2-cells in $\cC\ps B$ in the following way, yielding a 2-functor:
\begin{itemize}
 \item on 1-cells:
 \[
 \xymatrix@!=6ex{
 A \ar[dr]_f^{}="1" \ar[r]^t & A' \ar[d]^{f'} \ar@{}[dl]|(.15){}="2" \ar@{=>}"2";"1"_{\theta}^{\sim} & \mapsto & f/B \ar[dr]_{Lf}^{}="1" \ar[r]^{Lt} & f'/B \ar[d]^{Lf'} \ar@{}[dl]|(.15){}="2" \ar@{=>}"2";"1"_{1} \\
 & B & & & B
 }
 \]
 where $Lt$ is determined by the equations $d_0(Lt)=td_0$ and $\psi_{f'}(Lt)=\psi_f\cdot\theta d_0$;
 \item on 2-cells:
 \[
 \xymatrix@!=6ex{
 A \ar[dr]_f^{}="1" \ar@/^4ex/[r]^{t_1}_{}="3" \ar[r]_{t_2}^{}="4" \ar@{=>}"3";"4"^{\xi} & A' \ar[d]^{f'} \ar@{}[dl]|(.15){}="2" \ar@{=>}"2";"1"^{\theta_2} & \mapsto & f/B \ar[dr]_{Lf}^{}="1" \ar@/^4ex/[r]^{Lt_1}_{}="3" \ar[r]_{Lt_2}^{}="4" \ar@{=>}"3";"4"^{L\xi}  & f'/B \ar[d]^{Lf'} \ar@{}[dl]|(.15){}="2" \ar@{=>}"2";"1"^{1} \\
 & B & & & B
 }
 \]
 where $L\xi$ is determined by the equations $(Lf')(L\xi)=1$ and $d_0(L\xi)=\xi d_0$.
\end{itemize}

One can also define a pseudo 2-natural tranformation $u\colon 1_{\cC\ps B}\to L$ and a (strict) 2-natural transformation $m\colon L^2\to L$ as follows:
\begin{itemize}
 \item on objects: for each $f\colon A \to B$,
 \[
 \xymatrix@!=6ex{
 A \ar[dr]_f^{}="1" \ar[r]^-{u_f} & f/B \ar[d]^{Lf} \ar@{}[dl]|(.15){}="2" \ar@{=>}"2";"1"_{1} & & Lf/B \ar[dr]_{L^2f}^{}="1" \ar[r]^-{m_f} & f/B \ar[d]^{Lf} \ar@{}[dl]|(.15){}="2" \ar@{=>}"2";"1"_{1} \\
 & B & & & B
 }
 \]
 where $u_f$ and $m_f$ are determined by the equations
\[
\left\{
\begin{array}{l}
d_0u_f=1 \\
\psi_fu_f=1
\end{array}
\right.
\quad\mbox{and}\quad
\left\{
\begin{array}{l}
d_0m_f=d_0d_0 \\
\psi_fm_f=\psi_{Lf}\cdot\psi_fd_0
\end{array}
\right.
\]
respectively.
 \item on 1-cells: for each $(t,\theta)$ as above,
\[
\xymatrix{
A \ar[d]_{t} \ar[r]^-{u_f} & f/B \ar[d]^{Lt} \\
A' \ar[r]_{u_{f'}} \ar@{}[ur]|(.3){}="1" \ar@{}[ur]|(.7){}="2" \ar@{=>}"2";"1"_{u_{(t,\theta)}}^{\sim} & f'/B
}
\]
is uniquely determined by the equations $d_0u_{(t,\theta)}=1$ and $(Lf')u_{(t,\theta)}=\theta$, while $m_{(t,\theta)}=1$.
\end{itemize}
One can check that these data satisfy the properties
\begin{enumerate}
 \item[M1] $m\ast uL=1=m\ast Lu$;
 \item[M2] $m\ast mL=m\ast Lm$.
\end{enumerate}
Hence the triple $(L,u,m)$ forms a 2-monad on $\cC\ps B$, which is not strict only because of the invertible 2-cells $u_{(t,\theta)}$ for each $(t,\theta)$.

Let us now consider, for each $f\colon A \to B$, the 2-cell $\kappa_f\colon u_fd_0\to 1_{f/B}$ determined by the equations $d_0\kappa_f=1$ and $(Lf)\kappa_f=\psi_f$, and define
\[
\xymatrix{
f/B \ar@/^3ex/[rr]^{Lu_f}_{}="1" \ar@/_3ex/[rr]_{u_{Lf}}^{}="2" \ar@{=>}"1";"2"_{\lambda_f} & & Lf/B
}
\]
as the unique 2-cell such that $d_0\lambda_f=\kappa_f$ and $(L^2f)\lambda_f=1$.

It is tedious but straightforward to check that the collection of all $\lambda_f$ for each $f$ gives rise to a modification $\lambda\colon Lu\to uL$ between natural tranformations in $\cC\ps B$. Moreover, $\lambda$ satisfies the following properties:
\begin{enumerate}
 \item[KZ1] $\lambda\ast u=1$;
 \item[KZ2] $m\ast\lambda=1$;
 \item[KZ3] $m\ast Lm\ast\lambda L=1$.
\end{enumerate}
Since M1 (=KZ0) already holds, the data $(L,u,m,\lambda)$ give rise to a KZ-doctrine in the sense of Definition 1.1 in \cite{Kock}. In other words, this structure provides a lax-idempotent 2-monad.

Let us observe that, in fact, $L\colon \cC\ps B \to \cC\ps B$ factors through the inclusion of $\cC/B$ in $\cC\ps B$, and the 1-cell components of $u$ are such that $u_{(t,1)}=1$, so that the above 2-monad on $\cC\ps B$ restricts to a \emph{strict} 2-monad on $\cC/B$, which is also part of a KZ-doctrine by the same $\lambda$. We will adopt the same notation for both monads as far as no confusion arises.
\medskip

Like $L$, also the 2-functors $R$ and $I$ on $\cC\ps B$, defined by the corresponding comma squares in (\ref{diag:commas}), can be endowed with a structure of 2-monad, which is colax-idempotent in the case of $R$ and pseudo-idempotent in the case of $I$. In both cases, these structures restrict to strict 2-monads $(R,v,n,\rho)$ and $(I,i,l,\iota)$ on $\cC/B$.

One of the most important features of KZ-doctrines is that the corresponding (pseudo-)algebra structures are unique up to isomorphism for each object and they are characterized as right (pseudo-)inverse left adjoint to the unit component of the monad. Applying this observation and its dual to the special cases of the 2-functors $L$, $R$ and $I$ described above, one can characterize (pseudo-)fibrations (and dually opfibrations) and isofibrations in \cC.

%
\begin{Proposition} \label{prop:fib}
For a morphism $f\colon A \to B$ in \cC\ the following conditions are equivalent and define an internal \emph{fibration} (respectively \emph{pseudo-fibration}):
\begin{enumerate}
\item
\begin{itemize}
\item[(i)] For all $X$ in \cC, $\cC(X,f)\colon\cC(X,A)\to\cC(X,B)$ is a fibration (respectively pseudo-fibration) in \Cat;
\item[(ii)] for all $g\colon Y \to X$, the commutative square below is a morphism of fibrations (respectively pseudo-fibrations) in \Cat:
\[
\xymatrix{
\cC(X,A) \ar[r]^-{\cC(g,A)} \ar[d]_{\cC(X,f)} & \cC(Y,A) \ar[d]^{\cC(Y,f)} \\
\cC(X,B) \ar[r]_-{\cC(g,B)} & \cC(Y,B)
}
\]
\end{itemize}
\item $f$ admits a structure of pseudo-algebra for the 2-monad $R\colon \cC/ B\to\cC/ B$ (respectively $R\colon \cC\ps B\to\cC\ps B$);
\item The morphism $v_f\colon f \to Rf$ admits a right adjoint in $\cC/ B$ (respectively $\cC\ps B$);
\item (Chevalley criterion) The morphism $f_1\colon A/A\to B/f$, determined by the equations
\[
\left\{
\begin{array}{l}
(Rf)f_1=fd_0 \\
d_1f_1=d_1 \\
\varphi_ff_1=f\mu_A
\end{array}
\right.
\]
admits a right adjoint in \cC\ with counit an identity (respectively isomorphism).
\end{enumerate}
\end{Proposition}

In practice, given an internal fibration according to the above definition 2, it is convenient to fix a corresponding pseudo-algebra structure once and for all (which in \Cat\ means to fix a cleavage). Accordingly, throughout the paper, $\Fib_{\cC}(B)$ will denote the 2-category whose objects are pseudo-algebras for the monad $R\colon\cC/ B\to\cC/ B$, whose 1-cells are strict pseudo-algebra morphisms, and with the obvious 2-cells (we shall write just $\Fib(B)$ for $\cC=\Cat$).

\begin{Remark}
{\rm
The definition of internal fibration (resp. pseudo-fibration) in a representable 2-category appears in the form 2 of Proposition \ref{prop:fib} in the works of Street \cite{Street74,StreetBI}. The characterizations 1 and 3 in Proposition \ref{prop:fib} are well-known and already present in the literature (see, for example, \cite{Weber}). As for the Chevalley criterion, it was first proved by Gray \cite{Gray} for fibrations in \Cat, while an internal version of it (for opfibrations) appears in \cite[Proposition 9]{Street74}, asking for the unit to be an isomorphism. As the following example shows, such a condition does not characterize opfibrations. In fact, it characterizes pseudo-opfibrations (see 3.17 in \cite{StreetBI}). This is the reason why we consider the characterization 4.\ also for internal (strict) fibrations. Since we could not find a proof of the latter in the literature, we provide it in the appendix for the sake of completeness.
}
\end{Remark}

\begin{Example}{\em
Let $\cC=\Cat$ and consider any functor $f\colon 1\to B$, where $1$ is the terminal category and $B$ is the groupoid with two objects and exactly one isomorphism between them. It is easy to see that $1/1\cong 1$, $f/B\cong B$, and these isomorphisms make the induced functor $f_1\colon 1/1\to f/B$ of the dual of Proposition \ref{prop:fib} 4.\ isomorphic to $f$. By uniqueness, the terminal functor $t\colon B \to 1$ is the only possible left adjoint to $f$, and $tf=1_1$, hence one can choose $1_{tf}$ as a counit. On the other hand, $ft$ is not equal, but isomorphic to $1_B$, so that $f$ and $t$ are actually adjoint with unit an isomorphism. But $f$  is not an opfibration, just a pseudo-opfibration.
}
\end{Example}

The equivalent conditions of Proposition \ref{prop:fib} may be easily adapted to define \emph{internal opfibrations} (respectively pseudo-opfibrations), for example replacing, in 2.~and 3., the monad $R$ with the monad $L$. For the reader's convenience, throughout the paper, most of the results are stated in terms of fibrations. The corresponding results for opfibrations can be obtained accordingly.

Before giving the characterization of isofibrations, we point out the following property which is specific for iso-comma squares and will be useful later on.

\begin{Lemma} \label{lemma:w_equiv}
For each $f\colon A \to B$ in \cC, the 1-cells $i_f$ and $w_f$ form an adjoint equivalence with, in particular, $w_fi_f=1_A$. This yields also an equivalence in $\cC\ps B$ with $(i_f,1)$ and $(w_f,\omega_f)$ as adjoint pair.
\end{Lemma}

\begin{proof}
$w_fi_f=1_A$ by definition of $i_f$. Then let $\epsilon_f\colon i_fw_f\to 1\colon f/_{\cong}B\to f/_{\cong}B$ be the unique 2-cell such that $(If)\epsilon_f=\omega_f$ and $w_f\epsilon_f=1_{w_f}$. It is easy to check that the inverse $\epsilon_f^{-1}$ can be defined by the equations $(If)\epsilon_f^{-1}=\omega_f^{-1}$ and $w_f\epsilon_f^{-1}=1_{w_f}$, and this completes the proof of the first assertion (triangle identities follow easily). For the second one it suffices to recall that $\omega_fi_f=1_f$ and use the first part of the Lemma.
\end{proof}

\begin{Proposition} \label{prop:isofib}
For morphism $f\colon A \to B$ in \cC\ the following conditions are equivalent and define an internal \emph{isofibration}:
\begin{enumerate}
\item  For all $X$ in \cC, $\cC(X,f)\colon\cC(X,A)\to\cC(X,B)$ is an isofibration  in \Cat;
\item $f$ admits a structure of pseudo-algebra for the 2-monad $I\colon \cC/ B\to\cC/ B$;
\item the morphism $i_f\colon f \to If$ admits a right adjoint in $\cC/ B$.
\end{enumerate}
\end{Proposition}

\begin{Remark} \label{rem:monads}
{\em
\begin{enumerate}
\item The ``pseudo'' version of Proposition \ref{prop:isofib} does not make sense, of course. In fact, every functor is an isofibration up to isomorphism, so that the ``pseudo'' version of condition 1. is empty. This is reflected in the fact that, thanks to Lemma \ref{lemma:w_equiv}, each $f$ always admits a pseudo-algebra structure for $I\colon \cC\ps B \to \cC\ps B$, given by $(w_f,\omega_f)$. Condition 2.\ says that, when $f$ is an isofibration, one can replace $(w_f,\omega_f)$ with another left adjoint (the pseudo-algebra structure) to $(i_f,1)$, in order to obtain an adjoint equivalence which actually lives in $\cC/B$.
\item By the general theory of monads, for each $f$ in \cC, $Lf$, $Rf$ and $If$ are free algebras for the corresponding monads (whose structure is given by the multiplication component at $f$) and it is immediate to see that these are actually strict algebras, which represent split opfibrations (respectively fibrations and isofibrations) internally.
\end{enumerate}
}
\end{Remark}

We recall from \cite{Street74} the following properties, which will be used later on.

\begin{Lemma} \label{lem:pbfib}
Fibrations are pullback stable.
\end{Lemma}

\begin{Corollary} \label{cor:p0p1}
Given a comma square
\[
\xymatrix{
f/g \ar[d]_{\overline{g}} \ar[r]^-{\overline{f}} & C \ar[d]^{g} \\
A \ar[r]_-{f} \ar@{}[ur]|(.3){}="1" \ar@{}[ur]|(.7){}="2" \ar@{=>}"1";"2" & B
}
\]
then $\overline{f}$ is an opfibration and $\overline{g}$ is a fibration. In particular, the canonical morphisms $d_0$ and $d_1$ of diagram (\ref{diag:commas}) are a fibration and an opfibration respectively.
\end{Corollary}

Applying the observation 2.\ of Remark \ref{rem:monads} to the 2-functor $I$ we see that, for each $f\colon A \to B$, its image $If$ is actually an isofibration. Restricting to pseudo-fibrations, we get the following result.

\begin{Lemma} \label{lemma:pseudo.to.fib}
If $f\colon A \to B$ is an internal pseudo-fibration in \cC, then $If$ is an internal fibration.
\end{Lemma}

\begin{proof}
First recall from Lemma \ref{lemma:w_equiv} that $f$ and $If$ are equivalent as objects in $\cC\ps B$, hence $If$ is a pseudo-fibration, since $f$ is. Moreover, $If$ is an isofibration, by Remark \ref{rem:monads} 2. Then, by the representability of the notions involved, i.e.\ point 1.\ in Propositions \ref{prop:fib} and \ref{prop:isofib}, the thesis follows from the well-known fact that, in \Cat, any pseudo-fibration which is also an isofibration is actually a fibration.
\end{proof}

In fact, $I$ sends also morphisms (and 2-cells) of pseudo-fibrations to morphisms (and 2-cells) of fibrations, so that the following result holds.

\begin{Proposition} \label{prop:I.pres.fib}
The restriction of $I\colon\cC\ps B\to \cC\ps B$ to the full sub-2-category of pseudo-fibrations can be factorized as
\[
\xymatrix{
\PsFib(B) \ar@{^{(}->}[rr] \ar[d] & & \cC\ps B \ar[d]^{I} \\
\Fib(B) \ar@{^{(}->}[r] & \cC/B \ar[r] & \cC\ps B.
}
\]
\end{Proposition}

Let us observe that, when $f$ is a pseudo-fibration, the adjoint equivalence $(w_f,\omega_f)\dashv (i_f,1)$ actually lives in $\PsFib(B)$.

\section{Coinverters and coidentifiers in $\Fib(B)$}

From now on, let \cC\ be a finitely complete 2-category with coidentifiers and (strict) coinverters of reflexive 2-cells, whose definition we recall for the sake of completeness (the reader may refer to \cite{Kelly} for example).

\begin{Definition}
The coidentifier (coinverter) of a 2-cell $\alpha$ is a 1-cell $q$ such that: 
\begin{enumerate}
 \item $q\alpha$ is an identity (isomorphism);
 \item for any other 1-cell $f$ such that $f\alpha$ is an identity (isomorphism), there exists a unique 1-cell $t$ with $tq=f$;
 \item for any 2-cell $\beta\colon g \to h$ such that $g\alpha$ and $h\alpha$ are identities (isomorphisms), there exists a unique 2-cell $\gamma$ with $\gamma q=\beta$;
\end{enumerate}
\end{Definition}

In this paper we will consider in particular coidentifiers (coinverters) of identees (invertees). Given a 1-cell $f$, we denote with $(K,\kappa)$ its identee, where $\kappa$ is the 2-universal 2-cell making $f\kappa$ an identity. We denote with $(W,\omega)$ the invertee of $f$, where $\omega$ is the 2-universal 2-cell making $f\omega$ an isomorphism.

Later on, we will take advantage of the following results concerning isofibrations.
Recall that a morphism is called $f$-vertical if its image under $f$ is an identity.

%
\begin{Lemma} \label{lemma:psvert=vert}
Let $f$ be an isofibration and $\alpha$ a 2-cell such that $f\alpha$ is an isomorphism. Then $\alpha$ factorizes as $\alpha=\sigma\cdot\tau$, where $\tau$ is $f$-vertical and $\sigma$ is an isomorphism.
\end{Lemma}

\begin{proof}
Since $f$ is an isofibration, the isomorphism $f\alpha$ admits a cartesian lifting $\sigma$, which is an isomorphism, at the codomain of $\alpha$. $\tau$ is then the unique $f$-vertical factorization of $\alpha$ through $\sigma$.
\end{proof}

\begin{Corollary} \label{cor:isofib.cons=gpd(isofib)}
Let $f$ be an isofibration, then:
\begin{enumerate}
 \item $f$ is conservative if and only if its fibres are groupoids;
 \item the coinverter of the identee of $f$ coincides with the coinverter of its invertee.
\end{enumerate}
\end{Corollary}

\begin{proof}
Let $(W,\omega)$ and $(K,\omega c)$ be the invertee and the identee, respectively, of $f$:
\[
\xymatrix@C=6ex{
K \ar[r]^-c & W \ar@/_2ex/[r]_{w_1}^{}="1" \ar@/^2ex/[r]^{w_0}_{}="2" \ar@{=>}"2";"1"^{\omega} & A \ar[r]^-f & B\,.
}
\]
Since $f\omega$ is an isomorphism by definition, as in Lemma \ref{lemma:psvert=vert}, we can factorize $\omega$ as a composite $\omega=\sigma\cdot\tau$, where $\sigma$ is a cartesian lifting of $f\omega$, and $\tau$ the unique $f$-vertical comparison 2-cell. $\tau$ being vertical, there is a unique $c'\colon W\to I$ such that $\omega cc'=\tau$. So we have factorized $\omega$ as in the following diagram:
\[
\xymatrix@!=7ex{
W \ar[r]^-{c'} \ar@/_/[drr]_{w_1}^{}="3" & K \ar[r]^-c_{}="4"  \ar@{=>}"4";"3"^{\sigma} & W \ar@/_2ex/[d]_{w_1}^{}="1" \ar@/^2ex/[d]^{w_0}_{}="2" \ar@{=>}"2";"1"^{\omega} \\
& & A\,.
}
\]
It is now easy to see that $f$ is conservative, i.e.\ its invertee is an isomorphism, if and only if its fibres are groupoids, i.e.\ its identee is an isomorphism.

As for the second statement, it suffices to observe that the coinverter of the identee $\omega c$ coinverts also $\omega=\sigma\cdot\omega cc'$.
\end{proof}

Obviously the last result does not hold in general if $f$ is not an isofibration, as it is whitnessed by the non-constant functor from the arrow-category $\mathbf{2}$ to the groupoid $\mathbf{I}$ with two objects and two non-trivial arrows.

It is easy to check that identees and coidentifiers in $\cC/B$ are computed in \cC, and the same property holds for coinverters of identees.

\begin{Lemma} \label{lemma:inv/B}
The forgetful 2-functor $U\colon \cC/B \to \cC$ creates coinverters of identees.
\end{Lemma}

We are going to explore the behaviour of the monad $R$ with respect to these limits and colimits. Analogous results can be proved for the monad $L$.

\begin{Remark} \label{rem:R}
It is worth observing that the functor $R$ can be described by means of the following construction:
\[
\xymatrix@!=5ex{
B/f \ar[r]_{d_1^*f} \ar@/^3ex/[rr]^{Rf} \ar[d]_{d_1} & B/B \ar[r]_-{d_0} \ar[d]_{d_1} & B \ar@{=}[d] \\
A \ar[r]_-{f} & B \ar@{=}[r] \ar@{}[ur]|(.3){}="1" \ar@{}[ur]|(.7){}="2" \ar@{=>}"2";"1"^{\mu_B} & B.
}
\]
That is, $R=(d_0)_{!}d_1^*$, i.e.\ the composite of the change-of-base 2-functor along $d_1$ with the composition 2-functor with $d_0$, which is left 2-adjoint to $d_0^*$.
\end{Remark}

\begin{Lemma} \label{lemma:R.pres.id}
The monad $R\colon\cC/B\to\cC/B$ preserves identees.
\end{Lemma}

\begin{proof}
By rk \ref{rem:R}, the thesis follows from the fact that $d_1^*$ preserves limits, being a right adjoint, and $(d_0)_{!}$ preserves identees.
\end{proof}

\begin{Lemma}
The identee of a morphism $p\colon(A,f)\to(C,g)$ in $\KFib(B)$ can be computed as in \cC.
\end{Lemma}

\begin{proof}
Let
\begin{equation} \label{diag:identee}
\xymatrix@!=6ex{
K \ar@/_2ex/[dr]_{h} \ar@/^2ex/[r]^{k_0}_{}="1" \ar@/_2ex/[r]_{k_1}^{}="2" \ar@{=>}"1";"2"^{\kappa} & A \ar[d]_{f} \ar[r]^{p} & C \ar[dl]^{g} \\
& B
}
\end{equation}
be an identee diagram in $\cC/B$ (which means that $\kappa$ is also the identee of $p$ in \cC). Since $p$ is a morphism in $\KFib(B)$ and $R$ preserves identees by Lemma \ref{lemma:R.pres.id}, it is straightforward to prove that the adjunctions $v_f\dashv r_f$ and $v_g\dashv r_g$ in the diagram
\[
\xymatrix@!=6ex{
B/h \ar@{-->}@/^/[d]^{r_h} \ar@/^2ex/[r]^{Rk_0}_{}="1" \ar@/_2ex/[r]_{Rk_1}^{}="2" \ar@{=>}"1";"2"^{R\kappa} & B/f \ar@/^/[d]^{r_f} \ar[r]_-{Rp} & B/g \ar@/^/[d]^{r_g} \\
K \ar@{-->}@/^/[u]^{v_h}_{\dashv} \ar@/^2ex/[r]^{k_0}_{}="1" \ar@/_2ex/[r]_{k_1}^{}="2" \ar@{=>}"1";"2"^{\kappa} & A \ar[r]^-{p} \ar@/^/[u]_{\dashv}^{v_f} & C \ar@/^/[u]_{\dashv}^{v_g}
}
\]
induce an adjunction $v_h\dashv r_h$ by the universal property of the identees.
\end{proof}

Coinverters and coidentifiers may not be preserved by the monad $R$, however this happens to be true in some cases of interest which we will explore later on. So, for a given object $B$ in \cC, we shall consider the property
\begin{itemize}
\item[($\dagger$)] The monad $R\colon \cC/B\to \cC/B$ preserves coinverters and coidentifiers of identees.
\end{itemize}

\begin{Proposition} \label{prop:coinv}
Let $B$ be an object in \cC\ satisfying $(\dagger)$, $p\colon(A,f)\to(C,g)$ a morphism in $\KFib(B)$ and $\kappa$ its identee in \cC. Then the coinverter $q\colon A\to Q$ of $\kappa$ in \cC\ induces a factorization
\[
\xymatrix@!=6ex{
A \ar@/^3ex/[rr]^{p} \ar[dr]_{f} \ar[r]_-{q} & Q \ar[d]_(.4){gs} \ar[r]_-{s} & C \ar[dl]^{g} \\
& B
}
\]
of $p$ in $\KFib(B)$, and $q\colon(A,f)\to(Q,gs)$ is the coinverter of $\kappa$ in $\KFib(B)$.
\end{Proposition}

\begin{proof}
Recall that $\kappa$ is also the identee of $p$ in $\cC/B$ and consider the corresponding identee diagram (\ref{diag:identee}) in $\cC/B$. Since $p$ (and then $f$) coinverts $\kappa$, the morphisms $s$ and $gs$ in the factorization above are uniquely determined by the universal property of $q$, and this explains why $q$ is a coinverter of $\kappa$ in $\cC/B$.

Since $f$ is a fibration, the unit component $v_f\colon (A,f) \to (B/f,Rf)$ admits a right adjoint $r_f$ in $\cC/B$. We call $\eta_f$ and $\epsilon_f$ the corresponding unit and counit. Likewise, $v_g$ has a right adjoint $r_g$, and $pr_f=r_g(Rp)$ since $p$ is a morphism in $\KFib(B)$. Let us consider the following diagram:
\[
\xymatrix@!=6ex{
B/h \ar@/^/[d]^{r_h} \ar@/^2ex/[r]^{Rk_0}_{}="1" \ar@/_2ex/[r]_{Rk_1}^{}="2" \ar@{=>}"1";"2"^{R\kappa} & B/f \ar@/^/[d]^{r_f} \ar@/^3ex/[rr]^{Rp} \ar[r]_-{Rq} & B/(gs) \ar@{-->}@/^/[d]^{r_{gs}} \ar[r]_-{Rs} & B/g \ar@/^/[d]^{r_g} \\
K \ar@/^/[u]_{\dashv}^{v_h} \ar@/^2ex/[r]^{k_0}_{}="1" \ar@/_2ex/[r]_{k_1}^{}="2" \ar@{=>}"1";"2"^{\kappa} & A \ar@/_3ex/[rr]_-{p} \ar[r]^-{q} \ar@/^/[u]_{\dashv}^{v_f} & Q \ar@/^/[u]_{\dashv}^{v_{gs}} \ar[r]^-{s} & C \ar@/^/[u]_{\dashv}^{v_g}
}
\]
Now, $pr_f(R\kappa)=r_g(Rp)(R\kappa)=1$, hence $r_f(R\kappa)$ factors through $\kappa$ and $qr_f(R\kappa)$ is an isomorphism. By the assumption $(\dagger)$, $Rq$ is the coinverter of $R\kappa$, so there exists a unique $r_{gs}\colon B/(gs)\to Q$ such that $r_{gs}(Rq)=qr_f$. By the 2-dimensional universal property of the coinverters $q$ and $Rq$, one can prove that a unit $\eta_{gs}$ and a counit $\epsilon_{gs}$ are induced by $\eta_f$ and $\epsilon_f$ respectively, making $v_{gs}\dashv r_{gs}$ an adjoint pair in $\cC/B$, so that $gs$ is a fibration. As a consequence of this construction, $q$ turns out to be a morphism of fibrations over $B$.

It remains to show that for each $c\colon (A,f) \to (Y,y)$ in $\KFib(B)$ such that $c\kappa$ is an isomorphism, the unique comparison morphism $t$ in $\cC/B$, induced by the coinverter $q$ and such that $tq=c$, is actually a morphism in $\KFib(B)$. Let us denote by $r_{y}$ the $R$-pseudo-algebra structure on $y$, i.e.\ the right adjoint to $v_y$, and observe that the diagram
\[
\xymatrix@C=6ex{
B/(gs) \ar[d]_{r_{gs}} \ar[r]^-{Rt} & B/y \ar[d]^{r_{y}} \\
Q \ar[r]_-{t} & Y
}
\]
commutes since $Rq$ is a coinverter, then epimorphic, and precomposition with $Rq$ gives the commutative square presenting $c$ as a morphism of $R$-pseudo-algebras.

Finally, the fact that $s$ is in $\KFib(B)$ follows from the last argument, taking $c=p$.
\end{proof}


\begin{Corollary} \label{cor:fib.to.frac}
Let $B$ be an object in \cC\ satisfying $(\dagger)$, $f\colon A \to B$ a fibration in \cC, $\kappa$ its identee in \cC, and $q\colon A \to Q$ its coinverter in \cC. Then the unique comparison morphism $s\colon Q \to B$, such that $sq=f$, is a fibration and $q$ is the coinverter of $\kappa$ in $\KFib(B)$.
\end{Corollary}

\begin{proof}
Apply Proposition \ref{prop:coinv} to the morphism $f\colon(A,f)\to(B,1_B)$ in $\KFib(B)$.
\end{proof}

\begin{Proposition} \label{prop:coid}
Let $B$ be an object in \cC\ satisfying $(\dagger)$, $p\colon(A,f)\to(C,g)$ a morphism in $\KFib(B)$ and $\kappa$ its identee in \cC. Then the coidentifier $q'\colon A\to Q'$ of $\kappa$ in \cC\ induces a factorization
\[
\xymatrix@!=6ex{
A \ar@/^3ex/[rr]^{p} \ar[dr]_{f} \ar[r]_-{q'} & Q' \ar[d]_(.4){gs'} \ar[r]_-{s'} & C \ar[dl]^{g} \\
& B
}
\]
of $p$ in $\KFib(B)$, and $q'\colon(A,f)\to(Q',g's')$ is the coidentifier of $\kappa$ in $\KFib(B)$.
\end{Proposition}

\begin{proof}
The proof is the same as for Proposition \ref{prop:coinv}, replacing coinverters by coidentifiers.
\end{proof}

\begin{Corollary} \label{cor:fib.to.sfrac}
Let $B$ be an object in \cC\ satisfying $(\dagger)$, $f\colon A \to B$ a fibration in \cC, $\kappa$ its identee in \cC, and $q\colon A \to Q$ its coidentifier in \cC. Then the unique comparison morphism $s\colon Q \to B$, such that $sq=f$, is a fibration and $q$ is the coidentifier of $\kappa$ in $\KFib(B)$.
\end{Corollary}

%

%
%

In the cases we are interested in, which will be studied in the next section \ref{sec:casestudy}, the property $(\dagger)$ relies upon the exponentiability of split opfibrations in \cC. In this context, exponentiability is to be understood in a 2-categorical sense: a 1-cell $f$ is exponentiable if the change-of-base 2-functor along $f$ has a right 2-adjoint.

\begin{Lemma} \label{lem:exp.to.dag}
If for an object $B$ in \cC, the comma projection $d_1$ in the diagram
\[
\xymatrix{
B/B \ar[r]^-{d_0} \ar[d]_{d_1} & B \ar@{=}[d] \\
B \ar@{=}[r] \ar@{}[ur]|(.3){}="1" \ar@{}[ur]|(.7){}="2" \ar@{=>}"2";"1"^{\mu_B} & B
}
\]
is exponentiable, then the functor
\[
R\colon \cC/B \to \cC/B
\]
has a right 2-adjoint. As a consequence, $B$ satisfies the condition $(\dagger)$. In particular, this holds for any $B$ when split opfibrations in \cC\ are exponentiable.
\end{Lemma}

\begin{proof}
It is easy to see that the functor $R$ can be described by means of the following construction:
\[
\xymatrix@!=5ex{
B/f \ar[r]_{d_1^*f} \ar@/^3ex/[rr]^{Rf} \ar[d]_{d_1} & B/B \ar[r]_-{d_0} \ar[d]_{d_1} & B \ar@{=}[d] \\
A \ar[r]_-{f} & B \ar@{=}[r] \ar@{}[ur]|(.3){}="1" \ar@{}[ur]|(.7){}="2" \ar@{=>}"2";"1"^{\mu_B} & B.
}
\]
That is, $R=(d_0)_{!}d_1^*$, i.e.\ the composite of the change-of-base 2-functor along $d_1$ with the composition 2-functor with $d_0$, which is left 2-adjoint to $d_0^*$. Hence $R$ is left 2-adjoint to $\Pi_{d_1}d_0^*$, where $\Pi_{d_1}$ denotes the right 2-adjoint to $d_1^*$, which exists by assumption.
\end{proof}

\begin{Remark} \label{rem:dagger'}
{\rm
If instead of $(\dagger)$ we ask for
\begin{itemize}
 \item[($\dagger'$)] The monad $L\colon \cC/B\to \cC/B$ preserves coinverters and coidentifiers of identees,
\end{itemize}
then the results of \ref{prop:coinv} and \ref{prop:coid} hold with $\KFib(B)$ replaced by $\OpFib_{\cC}(B)$. Accordingly, if $d_0$ is exponentiable, and in particular when split fibrations are exponentiable in \cC, then $L$ admits a right 2-adjoint and $(\dagger')$ holds for $B$.
}
\end{Remark}

\subsection{Case study: \Cat\ and $\Fib(B)$} \label{sec:casestudy}

It is well-known that fibrations in \Cat\ are exponentiable \cite{Giraud} in the classical 1-categorical sense. As observed by Johnstone in \cite{Johnstone}, this property holds also in the 2-categorical sense recalled above. As a consequence, by Lemma \ref{lem:exp.to.dag} and Remark \ref{rem:dagger'}, we have:

\begin{Corollary} \label{cor:cat.dag}
In the 2-category \Cat, each object $B$ satisfies the conditions $(\dagger)$ and $(\dagger')$.
\end{Corollary}

One can extend the last property from \Cat\ to $\Fib(B)$ for each $B$, by means of the pseudo-functorial interpretation of fibrations in \Cat.

\begin{Remark} \label{rem:coinv.pseudofunct}
{\rm
As far as coinverters and coidentifiers of identees in $\Fib(B)$ are concerned, whose existence is guaranteed by Proposition \ref{prop:coinv} and Corollary \ref{cor:cat.dag}, let us observe that their construction can be performed fibrewise. In fact, given an identee
\[
\xymatrix{
K \ar@/^2ex/[r]_{}="1" \ar@/_2ex/[r]^{}="2" \ar@{=>}"1";"2"^{\kappa} & A
}
\]
in $\Fib(B)$, the collection of the coinverters (respectively coidentifiers) $q_b$ of its restrictions $\kappa_b$ to the fibres
\[
\xymatrix@!=6ex{
K_b \ar@/_2ex/[r]^{}="1" \ar@/^2ex/[r]_{}="2" \ar@{=>}"2";"1"^{\kappa_b} \ar[d]_{\beta^*} & A_b \ar[r]^{q_b} \ar[d]^{\beta^*} & Q_b \ar@{-->}[d]^{\beta^*} \\
K_{b'} \ar@/_2ex/[r]^{}="1" \ar@/^2ex/[r]_{}="2" \ar@{=>}"2";"1"^{\kappa_{b'}} & A_{b'} \ar[r]_{q_{b'}} & Q_{b'}
}
\]
gives rise to a natural transformation between pseudo-functors (see Section \ref{sec:pseudo} for details). It is easy to check that the corresponding morphism in $\Fib(B)$ is the coinverter (resp. coidentifier) of $\kappa$.
}
\end{Remark}

\begin{Lemma} \label{lem:internal.fib.pres}
For each internal fibration $p\colon(E,e)\to(A,a)$ in $\Fib(B)$, the change of base 2-functor $p^*\colon\Fib(B)/(A,a)\to\Fib(B)/(E,e)$ preserves coinverters and coidentifiers of identees.
\end{Lemma}

\begin{proof}
We will prove the result concerning internal fibrations and coinverters, the variations involving opfibrations and coidentifiers are obtained analogously.

Let the arrow $q\colon((C,c),f)\to((D,d),g)$ in the diagram
\[
\xymatrix@!=1.5ex{
K \ar@/_2ex/[ddrr]_{k} \ar@/^2.5ex/[rr]^{u}_{}="1" \ar@/_2.5ex/[rr]_{v}^{}="2" \ar@{=>}"1";"2"^{\kappa} & & C \ar[dd]_{c} \ar[rrr]^{f} \ar[dr]^{q} & & & A \ar@/^3ex/[ddlll]^{a} \\
& & & D \ar[urr]_{g} \ar[dl]_{d} \\
& & B
}
\]
be the coinverter in $\Fib(B)/(A,a)$ of an identee $\kappa$, and consider its image under the change of base 2-functor $p^*$, i.e.\ the upper part of the next diagram (we omit all arrows over $B$, all pullbacks provide in fact fibrations over $B$):
\[
\xymatrix@R=3ex@C=5ex{
K \times_A E \ar[dd] \ar@/^2.5ex/[rr]^{p^*u}_{}="1" \ar@/_2.5ex/[rr]_{p^*v}^{}="2" \ar@{=>}"1";"2"^{p^*\kappa} & & C \times_A E \ar[dd] \ar[rr]^{p^*f} \ar[dr]_{p^*q} & & E \ar[dd]^{p} \\
& & & D \times_A E \ar[ur]_{p^*g} \ar[dd] \\
K \ar@/^2.5ex/[rr]^{u}_{}="1" \ar@/_2.5ex/[rr]_{v}^{}="2" \ar@{=>}"1";"2"^{\kappa} & & C \ar[rr]^(.35){f} \ar[dr]_{q} & & A. \\
& & & D \ar[ur]_{g}
}
\]
We would like to show that $p^*q$ is the coinverter of the identee $p^*\kappa$ in $\Fib(B)/(E,e)$. To this end, we consider the restriction of the above diagram to the fibres over any object $b$ in $B$. By limit commutation, the latter is the same as the corresponding change of base diagram in the fibres over $b$:
\[
\xymatrix@R=3ex@C=5ex{
K_b \times_{A_b} E_b \ar[dd] \ar@/^2.5ex/[rr]^{p_b^*u_b}_{}="1" \ar@/_2.5ex/[rr]_{p_b^*v_b}^{}="2" \ar@{=>}"1";"2"^{p_b^*\kappa_b} & & C_b \times_{A_b} E_b \ar[dd] \ar[rr]^{p_b^*f_b} \ar[dr]_{p_b^*q_b} & & E_b \ar[dd]^{p_b} \\
& & & D_b \times_{A_b} E_b \ar[ur]_{p_b^*g_b} \ar[dd] \\
K_b \ar@/^2.5ex/[rr]^{u_b}_{}="1" \ar@/_2.5ex/[rr]_{v_b}^{}="2" \ar@{=>}"1";"2"^{\kappa_b} & & C_b \ar[rr]^(.35){f_b} \ar[dr]_{q_b} & & A_b. \\
& & & D_b \ar[ur]_{g_b}
}
\]
Now observe that, by Remark \ref{rem:coinv.pseudofunct}, $q_b$ is the coinverter of $\kappa_b$. Moreover, since $p_b$ is a fibration in \Cat\ by assumption, it is exponentiable, hence $p_b^*$ is a left 2-adjoint and $p_b^*q_b$ is the coinverter of $p_b^*\kappa_b$. Finally, again by Remark \ref{rem:coinv.pseudofunct}, $p^*q$ is the coinverter of $p^*\kappa$ in $\Fib(B)$, and hence in $\Fib(B)/(A,a)$.
\end{proof}

\begin{Proposition} \label{prop:dag.to.fib}
In the 2-category $\Fib(B)$, each object satisfies the conditions $(\dagger)$ and $(\dagger')$.
\end{Proposition}

\begin{proof}
Let $a\colon A \to B$ be a fibration of categories, then by Corollary \ref{cor:p0p1} the projections $d_0$ and $d_1$ of the comma square in $\Fib(B)$
\[
\xymatrix{
(A,a)/(A,a) \ar[r]^-{d_0} \ar[d]_{d_1} & (A,a) \ar@{=}[d] \\
(A,a) \ar@{=}[r] \ar@{}[ur]|(.3){}="1" \ar@{}[ur]|(.7){}="2" \ar@{=>}"2";"1"^{\mu_{(A,a)}} & (A,a)
}
\]
are an internal fibration and opfibration respectively. As a consequence, by Lemma \ref{lem:internal.fib.pres}, the corresponding change of base 2-functors $d_0^*$ and $d_1^*$ preserve coinverters and coidentifiers of identees. Now likewise in the proof of Lemma \ref{lem:exp.to.dag}, the thesis follows from the fact that $R=(d_0)_{!}d_1^*$ and $(d_0)_{!}$ is a left adjoint (and similarly for $L$).
\end{proof}

\begin{Proposition}
Let $p\colon(A,f)\to(C,g)$ be an internal fibration in $\Fib(B)$. Then the morphism $s$ in the factorization of Proposition \ref{prop:coinv} is an internal fibration in $\Fib(B)$.
\end{Proposition}

\begin{proof}
By Proposition \ref{prop:dag.to.fib}, we can apply Corollary \ref{cor:fib.to.frac} to the fibration $p$ in $\Fib(B)$.
\end{proof}

\section{Two factorization systems for (fibrewise) opfibrations in $\Fib(B)$}

\subsection{Two factorization systems in \Cat}

Let us consider the diagram
\[
\xymatrix@!=6ex{
W \ar@/^2ex/[r]_{}="1" \ar@/_2ex/[r]^{}="2" \ar@{=>}"1";"2"^{\omega} & A \ar[d]_{f} \ar[r]^{q} & Q \ar[dl]^{s} \\
& B
}
\]
in \Cat, where $q$ is the coinverter of the invertee $\omega$ of $f$.

The comparison functor $s$ is not conservative in general. One has to repeat this ``invertee-coinverter'' procedure possibly infinitely many times in order to get a conservative comparison, and an actual factorization system in \Cat\ \cite[Theorem C.0.31]{Joyal}. A similar phenomenon occurs when taking the ``identee-coidentifier'' analogue of the previous procedure, which allows to factor any functor as a (possibly infinite) sequence of coidentifiers followed by a discrete functor, i.e.~a functor whose fibres are discrete, yielding another factorization system in \Cat.

We are going to show that if we restrict ourselves to fibrations, both factorization procedures above simplify to a single-step factorization. Let us start with the second one showing that, for fibrations, it coincides with the comprehensive factorization system introduced in \cite{StW}. This actually holds not only in \Cat, but in any 2-category $\Cat(\cE)$ of internal categories where the construction of the comprehensive factorization of any functor provided in \cite{StV} is still valid, as, for example, when \cE\ is a finitely cocomplete locally cartesian closed category, like any topos \cE.

\begin{Proposition} \label{prop:comprehensive.cat}
Let $f\colon A \to B$ be a fibration in $\Cat(\cE)$ as above. The coidentifier of the identee of $f$, together with the comparison functor, factorizes $f$ into a final functor followed by a discrete fibration, giving then the comprehensive factorization of $f$. Starting with $f$ opfibration, the same procedure yields the dual comprehensive factorization of $f$ given by an initial functor followed by a discrete opfibration.
\end{Proposition}

\begin{proof}
We consider just the case of fibrations.
Following the approach of Section 3 in \cite{StV}, we perform the comprehensive factorization of $f$ by taking the free $R$-algebra $Rf$, which is a split fibration, and then reflecting it into a discrete fibration $p$:
\[
\xymatrix@!=7ex{
K_f \ar@/_/[r]_{\bar v}^{\top} \ar@/_2ex/[d]^{}="1" \ar@/^2ex/[d]_{}="2" \ar@{=>}"1";"2"^{\kappa_f} & K_{Rf} \ar@/_/[l]_{\bar r} \ar@/_2ex/[d]^{}="1" \ar@/^2ex/[d]_{}="2" \ar@{=>}"1";"2"_{\kappa_{Rf}} & Q \\
A \ar@/_/[r]_{v}^{\top} \ar@/^/[rru]^(.75){q} \ar[d]_{f} & B/f \ar[d]_{Rf} \ar@/_/[l]_{r} \ar[r]_-d & \pi_{0B}(B/f) \ar[u]^{t} \ar[d]_{p} \\
B \ar@{=}[r] & B \ar@{=}[r] & B.
}
\]
By construction of the above reflection, $d$ is the coidentifier of the identee $\kappa_{Rf}$ of $Rf$. Considering the adjunction $v \dashv r$ provided by the fact that $f$ is a fibration, we get $dv\kappa_f=d\kappa_{Rf}\bar{v}=1$, where $\kappa_f$ is the identee of $f$. Let now $q$ be a functor such that $q\kappa_f=1$, and consider the unit $\eta\colon 1 \to rv$ of the adjunction $v\dashv r$ in $\Cat(\cE)/B$. Then $f\eta=1$ and $\eta$ is contained in $\kappa_f$, so that $q\eta=1$ as well, and $qrv=q$. On the other hand, the counit $\epsilon\colon vr\to 1$ is such that $(Rf)\epsilon=1$, so it is contained in $\kappa_{Rf}$ and hence $d\epsilon=1$ and $dvr=d$.

Now, $qr\kappa_{Rf}=q\kappa_f\bar{r}=1$, so by the universal property of the coidentifier $d$ there exists a unique $t$ such that $td=qr$. Hence $q=qrv=tdv$, and $t$ is unique with this last property. Indeed, if $t'dv=q$ for some $t'$, then $t'd=t'dvr=qr=td$ and hence $t'=t$ since $d$ is cancellable. This proves that $dv$ is the coidentifier of $\kappa_f$, and then it is final \cite{StV}.
\end{proof}

\begin{Remark}{\rm
The last result also shows that, when $f$ is a fibration, the factorization of $f$ given by (sequence of coidentifiers, discrete functor) reduces to a single coidentifier and coincides with the comprehensive factorization.
}
\end{Remark}

In the special case of \Cat, Proposition \ref{prop:comprehensive.cat} can be proved directly by means of the pseudo-functorial interpretation of fibrations. This indeed is what we are going to do in order to obtain the analogous result, where coidentifiers are replaced by coinverters and discrete fibrations are replaced by fibrations in groupoids (i.e.~fibrations whose identee is an isomorphism).

\begin{Proposition} \label{prop:isocompr.cat}
Each fibration (respectively opfibration) $f\colon A \to B$ in \Cat\ admits a factorization given by the coinverter of the identee of $f$ followed by a fibration (resp.~opfibration) in groupoids. This factorization of $f$ coincides with the one given by (sequence of coinverters, conservative functor).
\end{Proposition}

\begin{proof}
Let us denote by
\[
\xymatrix{
\Cat \ar@<1.2ex>[r]^{\pi}_{\bot} & \Gpd \ar@{_(->}@<1.2ex>[l]^{i}
}
\]
the reflection of categories in groupoids, where the left 2-adjoint $\pi$ can be obtained by taking as unit component, for each category $A$, the coinverter $\eta_A$ of the 2-cell $\mu_A$ associated with the comma category $A/A$:
\[
\xymatrix@C=6ex{
A/A \ar@/_2ex/[r]^{}="1" \ar@/^2ex/[r]_{}="2" \ar@{=>}"2";"1"^{\mu_A} & A \ar[r]^-{\eta_A} & \pi(A)\,.
}
\]
Consider now a fibration $f\colon A\to B$ and denote by $[f]\colon B^{\op}\to \Cat$ the corresponding pseudo-functor. The composite $\pi[f]\colon B^{\op}\to \Gpd$ gives rise to a fibration in groupoids $\overline{f}\colon\overline{A}\to B$. On the other hand, $\eta[f]$ corresponds to a morphism $q\colon (A,f) \to (\overline{A},\overline{f})$ in $\Fib(B)$:
\[
\xymatrix{
A \ar[dr]_f \ar[rr]^{q} & & \overline{A} \ar[dl]^{\overline{f}} \\
& B.
}
\]
The component $q_b$ of $\eta[f]$ at an object $b$ of $B$ is actually the coinverter $\eta_{A_b}$ of $\mu_{A_b}$:
\[
\xymatrix@C=6ex{
A_b/A_b \ar@/_2ex/[r]^{}="1" \ar@/^2ex/[r]_{}="2" \ar@{=>}"2";"1"^{\mu_{A_b}} & A_b \ar[r]^-{\eta_{A_b}} & \pi(A_b).
} 
\]
It is not difficult to see that the pair $(A_b/A_b,\mu_{A_b})$ coincides with the restriction $(K_b,\kappa_b)$ of the identee $(K,\kappa)$ of $f$ to the fibre over $b$. Hence, as explained in Remark \ref{rem:coinv.pseudofunct}, $q$ turns out to be the coinverter of $\kappa$ in $\Fib(B)$. Thanks to Corollary \ref{cor:isofib.cons=gpd(isofib)}, $\overline{f}$ is conservative and we get the desired factorization of $f$.
\end{proof}

\subsection{From \Cat\ to $\Fib(B)$} \label{sec:pseudo}

We are going to use now the results of the previous section to produce analogous factorizations for (fibrewise) opfibrations in $\Fib(B)$. Our focus on this case is motivated by the study of cohomology theories provided in \cite{CMMV17}. Analogous results are still valid while considering (fibrewise) fibrations in $\Fib(B)$ or (fibrewise) opfibrations in $\OpFib(B)$. It is a well known result, due to B\'enabou, that fibrations in $\Fib(B)$ (opfibrations in $\OpFib(B)$) are just fibrations (opfibrations) in \Cat, so that this case reduces to Propositions \ref{prop:comprehensive.cat} and \ref{prop:isocompr.cat}, thanks to Propositions \ref{prop:coinv} and \ref{prop:coid}.

Let us recall from \cite{CMMV17} the following definitions and results.

\begin{Definition} \label{def:fw} {\rm(see \cite[Definition 2.1]{CMMV17})}
We say that a morphism $p\colon (A,f)\to (C,g)$ in $\Fib(B)$ is a \emph{fibrewise (discrete) opfibration} if, for every object $b$ of $B$, the restriction $p_b\colon A_b\to C_b$ of $p$ to the $b$-fibres is a (discrete) opfibration.
\end{Definition}

From Theorem 2.8 in \cite{CMMV17} it follows that every internal opfibration in $\Fib(B)$ is a fibrewise opfibration, while the latter is exactly a morphism in $\Fib(B)$ which is an internal opfibration in $\Cat/B$ (see Proposition 2.5 and Proposition 2.7 in \cite{CMMV17}). By Corollary 2.9 in [\emph{loc.~cit.}] in the discrete case the two notions coincide. Recall also from [\emph{loc.~cit.}] that Yoneda's \emph{regular spans} and \emph{two-sided fibrations} are instances, respectively, of fibrewise opfibrations and internal opfibrations in $\Fib(B)$. 
\medskip

Let us consider a fibrewise opfibration $p\colon(A,f)\to(C,g)$ in $\Fib(B)$ and focus our attention on its restriction $p_b$ to a single fibre over some $b$ in $B$. By the dual of Proposition \ref{prop:comprehensive.cat}, we can perform the comprehensive factorization of the opfibration $p_b\colon A_b \to C_b$ by means of the coidentifier $q_b$ of its identee:
\[
\xymatrix{
A_b \ar@/_3ex/[rr]_{p_b} \ar[r]^{q_b} & Q_b \ar[r]^{s_b} & C_b.
}
\]
Since $f$ and $g$ are fibrations, the assignments $b\mapsto A_b$ and $b\mapsto C_b$ are pseudo-functorial and the collection of the functors $p_b$ gives rise to a natural transformation of pseudo-functors from $B^{op}$ to \Cat. By the universal property of the coidentifiers $q_b$ for each $b$, the $Q_b$'s are also pseudo-functorial and the $q_b$'s and $s_b$'s organize in two natural transformations. Let us briefly show how this can be proved.

For each $b$ in $B$, we denote by $(K_b,\kappa_b)$ the identee of $p_b$. Let us observe that also the assignment $b\mapsto K_b$ is pseudo-functorial and together with the collection of the $\kappa_b$'s, it determines the identee $(K,\kappa)$ of the cartesian functor $p$. Given an arrow $\beta\colon b'\to b$ in $B$, we always denote by $\beta^*$ its associated change of base functor for any chosen fibration over $B$. Since $q_{b'}\beta^*\kappa_b=q_{b'}\kappa_{b'}\beta^*=1$, by the universal property of the coidentifier $q_b$ there is a unique functor $\beta^*\colon Q_b\to Q_b'$ such that $\beta^*q_b=q_{b'}\beta^*$:
\[
\xymatrix@!=6ex{
K_b \ar@/_2ex/[r]^{}="1" \ar@/^2ex/[r]_{}="2" \ar@{=>}"2";"1"^{\kappa_b} \ar[d]_{\beta^*} & A_b \ar[r]^{q_b} \ar[d]^{\beta^*} & Q_b \ar[r]^{s_b} \ar@{-->}[d]^{\beta^*} & C_b \ar[d]^{\beta^*} \\
K_{b'} \ar@/_2ex/[r]^{}="1" \ar@/^2ex/[r]_{}="2" \ar@{=>}"2";"1"^{\kappa_{b'}} & A_{b'} \ar[r]_{q_{b'}} & Q_{b'} \ar[r]_{s_{b'}} & C_{b'}.
}
\]
Given a composable pair of arrows
\[
\xymatrix{
b'' \ar[r]^{\beta'} & b' \ar[r]^{\beta} & b
}
\]
in $B$, let $\phi_{\beta,\beta'}\colon (\beta')^*\beta^*\to (\beta\beta')^*$ and $\gamma_{\beta,\beta'}\colon (\beta')^*\beta^*\to (\beta\beta')^*$ be the corresponding coherence isomorphisms induced by the fibrations $f$ and $g$ respectively. Since $q_{b''}\phi_{\beta,\beta'}$ is a 2-cell between $q_{b''}(\beta')^*\beta^*=(\beta')^*\beta^*q_b$ and $q_{b''}(\beta\beta')^*=(\beta\beta')^*q_b$, then by the universal property of the coidentifier $q_b$ there exists a unique invertible 2-cell $\psi_{\beta,\beta'}$ such that $\psi_{\beta,\beta'}q_b=q_{b''}\phi_{\beta,\beta'}$.
\[
\xymatrix@C=10ex@R=8ex{
A_b \ar[r]^{q_b} \ar[d]^{\beta^*}_(.8){}="1" \ar@/_6.5ex/[dd]_{(\beta\beta')^*}^(.4){}="2" \ar@{=>}"1";"2"_{\phi_{\beta,\beta'}}^\sim & Q_b \ar[r]^{s_b} \ar[d]_{\beta^*}^(.8){}="3" \ar@/^6.5ex/[dd]^(.25){(\beta\beta')^*}^(.4){}="4" \ar@{=>}"3";"4"^{\psi_{\beta,\beta'}}_\sim & C_b \ar[d]_{\beta^*}^(.8){}="5" \ar@/^6.5ex/[dd]^{(\beta\beta')^*}^(.4){}="6" \ar@{=>}"5";"6"^{\gamma_{\beta,\beta'}}_\sim \\
A_{b'} \ar[r]^{q_{b'}} \ar[d]^{(\beta')^*} & Q_{b'} \ar[r]|(.42)\hole^(.7){s_{b'}} \ar[d]_{(\beta')^*} & C_{b'} \ar[d]_{(\beta')^*} \\
A_{b''} \ar[r]_{q_{b''}} & Q_{b''} \ar[r]_{s_{b''}} & C_{b''}.
}
\]
Moreover, the equality $\gamma_{\beta,\beta'}s_bq_b=s_{b''}q_{b''}\phi_{\beta,\beta'}=s_{b''}\psi_{\beta,\beta'}q_b$ implies that $\gamma_{\beta,\beta'}s_b=s_{b''}\psi_{\beta,\beta'}$ since $q_b$ is right cancellable. Finally, the coherence conditions on the $\psi$'s making the assignment $b\mapsto Q_b$ into a pseudo-functor can be deduced once again by the universal property of the $q_b$'s. If $(K,\kappa)$ is the identee of $p$, since, for a morphism $t\colon(A,f)\to(Y,y)$ in $\Fib(B)$, $t\kappa=1$ if and only if $t_b\kappa_b=1$ for each $b$ in $B$, $q$ is actually the coidentifier of $\kappa$ in $\Fib(B)$. Recalling that, since each $s_b$ is a discrete opfibration, $s$ is an internal discrete opfibration, we have just proved the following result.

\begin{Proposition} \label{prop:comprehensive.int}
Every fibrewise opfibration $p\colon(A,f)\to(C,g)$ in $\Fib(B)$ admits a \emph{comprehensive factorization}
\[
\xymatrix@!=6ex{
A \ar@/^3ex/[rr]^{p} \ar[dr]_{f} \ar[r]_-{q} & Q \ar[d]_(.4){h} \ar[r]_-{s} & C \ar[dl]^{g} \\
& B
}
\]
where $q$ is the coidentifier of the identee of $p$ and $s$ is a discrete opfibration in $\Fib(B)$.
\end{Proposition}

The same result (with $g$ a split fibration) is obtained in Section 3.3 of \cite{CMMV17}, by providing an explicit construction of the discrete opfibration $s$ together with an ad hoc definition of $q$, which is later on proved to be the coidentifier of the identee of $p$.
\smallskip

As a corollary of Proposition \ref{prop:comprehensive.int}, we get an extension of Proposition \ref{prop:comprehensive.cat} in the case of \Cat.

\begin{Corollary}
For every fibrewise opfibration $p\colon(A,f)\to(C,g)$ in $\Fib(B)$, the factorization of Proposition \ref{prop:comprehensive.int} coincides with the one given by (sequence of coidentifiers, discrete functor) in \Cat.
\end{Corollary}

The present approach allows us to obtain an analogous result concerning the factorization given by (coinverter, opfibration in groupoids).

\begin{Proposition} \label{prop:refl.gpd.fib}
Every fibrewise (resp.~internal) opfibration $p\colon(A,f)\to(C,g)$ in $\Fib(B)$ admits a factorization
\[
\xymatrix@!=6ex{
A \ar@/^3ex/[rr]^{p} \ar[dr]_{f} \ar[r]_-{q'} & Q' \ar[d]_(.4){h'} \ar[r]_-{s'} & C \ar[dl]^{g} \\
& B
}
\]
where $q'$ is the coinverter of the identee of $p$ and $s'$ is a fibrewise (resp.~internal) opfibration in groupoids in $\Fib(B)$.
\end{Proposition}

\begin{proof}
Let us start with a fibrewise opfibration $p$. The coinverter $q'$ of the identee of $p$ in $\Fib(B)$ and the comparison arrow $s'$ can be constructed following the lines of the previous paragraph, by means of coinverters taken fibrewise. This means that for each $b$ in $B$, $(q'_b,s'_b)$ gives the factorization of $p_b$ into a coinverter followed by an opfibration in groupoids, thanks to Proposition \ref{prop:isocompr.cat}. Hence $s'$ is a fibrewise opfibration in groupoids.

If moreover $p$ is an internal opfibration, then $s'$ is also an internal opfibration by Corollary \ref{cor:fib.to.frac} applied to $\Fib(B)$, thanks to Proposition \ref{prop:dag.to.fib}.
\end{proof}

As a consequence, we get an extension of Proposition \ref{prop:isocompr.cat}.

\begin{Proposition} \label{prop:ultima}
For every fibrewise opfibration $p\colon(A,f)\to(C,g)$ in $\Fib(B)$, the factorization of Proposition \ref{prop:refl.gpd.fib} coincides with the one given by (sequence of coinverters, conservative functor) in \Cat.
\end{Proposition}

\begin{proof}
By Proposition \ref{prop:coinv}, the arrow $q'$ in the above factorization can be obtained as the coinverter of the identee $\kappa$ of $p$ in \Cat\ (which is also the identee in $\Cat/B$). Actually, $q'$ is also the coinverter of the invertee $\omega$ of $p$ in \Cat. Indeed, since $f\omega$ is an isomorphism, thanks to Lemma \ref{lemma:psvert=vert}, it factorizes as $\sigma\cdot\tau$, with $\sigma$ an isomorphism and $\tau$ an $f$-vertical 2-cell. Then $p\tau=p(\sigma^{-1}\cdot\omega)$ is an isomorphism, hence $\tau$ factorizes through the invertee $\omega'$ of $p$ in $\Cat/B$.
Since $p$ is an opfibration in $\Cat/B$, by Corollary \ref{cor:isofib.cons=gpd(isofib)}, the coinverter $q'$ of $\kappa$ in $\Cat/B$ is also the coinverter of $\omega'$ in $\Cat/B$, hence $q'\tau$ is an isomorphism. As a consequence $q'\omega$ is an isomorphism, and the universal property in \Cat~follows easily.

Using the same technique as before, considering now the invertee $\omega$ of $s'$ in \Cat, we get that $\omega=\sigma\cdot\tau$ with $\sigma$ an isomorphism and $\tau$ an $h'$-vertical 2-cell. As above, $\tau$ factorizes through the invertee $\omega'$ of $s'$ in $\Cat/B$, which is an isomorphism by Corollary \ref{cor:isofib.cons=gpd(isofib)}, because $s'$ is an opfibration in groupoids in $\Cat/B$. In conclusion, the factorization of $p$ in \Cat\ through a conservative functor is obtained by means of just one coinverter $q'$.
\end{proof}


\appendix

\section{Proof of the Chevalley criterion}

Here we provide a detailed proof of the Chevalley criterion for internal fibrations in a finitely complete 2-category \cC. We assume 1.\ of Proposition \ref{prop:fib} as a definition of internal fibration in \cC\ (that we make explicit in the next lines), and we prove the equivalence between 1.\ and 4.

It is not hard to see, following \cite{Weber}, that an equivalent formulation of 1.\ is given by the following. For all $\beta\colon b \to fa$ there exists an $f$-cartesian 2-cell over $\beta$, i.e.\ $\alpha\colon a'\to a$ such that
\begin{enumerate}
\item $f\alpha=\beta$;
\item for all $g\colon Y \to X$, all
\[
\xymatrix@!=7ex{
Y \ar@/^2ex/[r]^{y}_{}="1" \ar@/_2ex/[r]_{ag}^{}="2" \ar@{=>}"1";"2"^{\sigma} & A
}
\]
and all $\gamma\colon fy\to bg$ such that $f\sigma=\beta g\cdot\gamma$
\[
\xymatrix@!=7ex{
Y \ar[r]^g \ar@/_1ex/[drr]_{fy}^{}="3" & X \ar[r]^a \ar[dr]_b^{}="1"_(.2){}="4" \ar@{=>}"3";"4"^{\gamma} & A \ar[d]^f_(.1){}="2" \ar@{=>}"1";"2"^(.4){\beta} \\
& & B
}
\]
there exists a unique $\tau\colon y \to a'g$ such that
\begin{itemize}
\item $\alpha g\cdot\tau=\sigma$
\[
\xymatrix@!=7ex{
Y \ar@/^5ex/[rr]^y_(.4){}="1" \ar[r]_{g}^(.8){}="2" \ar@{=>}"1";"2"_{\tau} & X \ar[r]^{a'}_{}="3" \ar@/_4ex/[r]_{a}^{}="4" \ar@{=>}"3";"4"_{\alpha} & A
}
\]
\item $f\tau=\gamma$.
\end{itemize}
\end{enumerate}

\begin{Proposition}[Chevalley criterion]
A 1-cell $f$ in \cC\ is a fibration if and only if the 1-cell $\overline{f}$ below
\[
\xymatrix{
A/A \ar@/_2ex/[ddr]_{fd_0} \ar@/^2ex/[drr]^{d_1} \ar[dr]^(.6){\overline{f}} \\
& B/f \ar[d]_{Rf} \ar[r]^-{d_1} & A \ar[d]^{f} \\
& B \ar[r]_{1_B} \ar@{}[ur]|(.3){}="1" \ar@{}[ur]|(.7){}="2" \ar@{=>}"1";"2"_{\varphi} & B
}
\]
uniquely determined by the equations
\[
\left\{
\begin{array}{l}
(Rf)\overline{f}=fd_0 \\
d_1\overline{f}=d_1 \\
\varphi\overline{f}=f\mu_A \\
\end{array}
\right.
\]
has a right adjoint with counit an identity.
\end{Proposition}

\begin{proof}
``only if'' part.

Suppose $f$ is a fibration, then the 2-cell $\varphi\colon Rf\to fd_1$ admits a cartesian lifting, i.e.\ an $f$-cartesian 2-cell $\alpha\colon a\to d_1$ such that $f\alpha=\varphi$. This $\alpha$ induces an arrow $r$ (our candidate right adjoint) as in the diagram
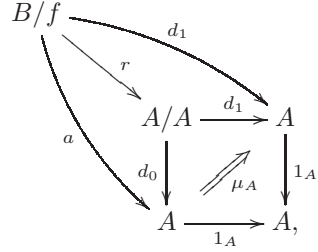
\[
\xymatrix{
B/f \ar@/_2ex/[ddr]_{a} \ar@/^2ex/[drr]^{d_1} \ar[dr]^(.6){r} \\
& A/A \ar[d]_{d_0} \ar[r]^-{d_1} & A \ar[d]^{1_A} \\
& A \ar[r]_{1_A} \ar@{}[ur]|(.3){}="1" \ar@{}[ur]|(.7){}="2" \ar@{=>}"1";"2"_{\mu_A} & A,
}
\]
uniquely determined by the equations
\[
\left\{
\begin{array}{l}
d_0r=a \\
d_1r=d_1 \\
\mu_Ar=\alpha. \\
\end{array}
\right.
\]
From the equality $\varphi\overline{f}r=f\mu_Ar=f\alpha=\varphi$, by the universal property of the comma object $B/f$, $\overline{f}r=1_{B/f}$. So we can choose $\epsilon=1_{1_{B/f}}$ as a counit of the desired adjunction $\overline{f}\dashv r$.

Let us now construct the corresponding unit. The 2-cell $\mu_A$ is such that $f\mu_A=\varphi\overline{f}=f\alpha\overline{f}$, so by cartesianness of $\alpha$ there exists a unique $\tau\colon d_0\to a\overline{f}$ such that
\[
\left\{
\begin{array}{l}
\alpha\overline{f}\cdot\tau=\mu_A \\
f\tau=1_{(Rf)\overline{f}} \\
\end{array}
\right.
\]
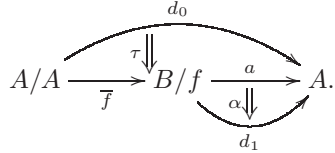
\[
\xymatrix@!=7ex{
A/A \ar@/^5ex/[rr]^{d_0}_(.4){}="1" \ar[r]_{\overline{f}}^(.8){}="2" \ar@{=>}"1";"2"_{\tau} & B/f \ar[r]^{a}_{}="3" \ar@/_4ex/[r]_{d_1}^{}="4" \ar@{=>}"3";"4"_{\alpha} & A.
}
\]
Let us now consider the following diagram
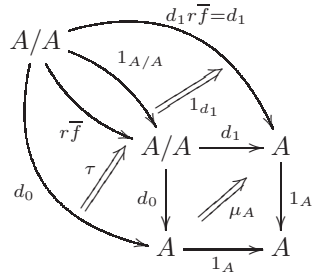
\[
\xymatrix{
A/A \ar@/_6ex/[ddr]_{d_0}^(.7){}="3" \ar@/^6ex/[drr]^{d_1r\overline{f}=d_1}_(.7){}="2" \ar@/_2ex/[dr]_{r\overline{f}}_(.8){}="4" \ar@/^2ex/[dr]^{1_{A/A}}^(.8){}="1" \ar@{=>}"1";"2"_{1_{d_1}} \ar@{=>}"3";"4"^{\tau} \\
& A/A \ar[d]_{d_0} \ar[r]^-{d_1} & A \ar[d]^{1_A} \\
& A \ar[r]_{1_A}  \ar@{}[ur]|(.3){}="1" \ar@{}[ur]|(.7){}="2" \ar@{=>}"1";"2"_{\mu_A} & A
}
\]
Since $\mu_Ar\overline{f}\cdot\tau=\alpha\overline{f}\cdot\tau=\mu_A$, by the 2-dimensional universal property of $A/A$, there exists a unique 2-cell $\eta\colon 1_{A/A}\to r\overline{f}$ such that
\[
\left\{
\begin{array}{l}
d_0\eta=\tau \\
d_1\eta=1_{d_1}. \\
\end{array}
\right.
\]
It remains to prove that $(\overline{f},r,\eta,\epsilon)$ form an adjunction in \cC, i.e.\ they satisfy the triangle identities. But $r\epsilon$ and $\epsilon\overline{f}$ being identities, since $\epsilon$ is, we only have to prove that $\eta r$ and $\overline{f}\eta$ are identities too. Consider first the diagram
\[
\xymatrix{
A/A \ar@/_3ex/[dr]_{\overline{f}}^{}="1" \ar@/^3ex/[dr]^{\overline{f}r\overline{f}=\overline{f}}_{}="2" \ar@{=>}"1";"2"^{\overline{f}\eta} \\
& B/f \ar[d]_{Rf} \ar[r]^-{d_1} & A \ar[d]^{f} \\
& A \ar[r]_{1_A}  \ar@{}[ur]|(.3){}="1" \ar@{}[ur]|(.7){}="2" \ar@{=>}"1";"2"_{\varphi} & B.
}
\]
We have that
\[
\left\{
\begin{array}{l}
(Rf)\overline{f}\eta=fd_0\eta=f\tau=1_{(Rf)\overline{f}} \\
d_1\overline{f}\eta=d_1\eta=1_{d_1} \\
\end{array}
\right.
\]
Hence by universal property $\overline{f}\eta=1_{\overline{f}}$. A similar argument applied to the diagram below will give us that $\eta r=1_r$.
\[
\xymatrix{
B/f \ar@/_3ex/[dr]_{r}^{}="1" \ar@/^3ex/[dr]^{r\overline{f}r=r}_{}="2" \ar@{=>}"1";"2"^{\eta r} \\
& A/A \ar[d]_{d_0} \ar[r]^-{d_1} & A \ar[d]^{1_A} \\
& A \ar[r]_{1_A}  \ar@{}[ur]|(.3){}="1" \ar@{}[ur]|(.7){}="2" \ar@{=>}"1";"2"_{\mu_A} & A.
}
\]
We have that
\[
\left\{
\begin{array}{l}
d_0\eta r=\tau r \\
d_1\eta r=1_{d_1}. \\
\end{array}
\right.
\]
Hence by universal property $\eta r=1_r$ if and only if $\tau r=1_{a}$. We are going to prove the latter by cartesianness of $\alpha$. In fact, $\tau r$ is such that
\[
\left\{
\begin{array}{l}
\alpha\cdot\tau r=\alpha\overline{f}r\cdot\tau r=\mu_A r=\alpha \\
f\tau r=1_{Rf} \\
\end{array}
\right.
\]
But the same properties are shared by $1_a$, so by uniqueness $\tau r=1_{a}$.

``if'' part.

Suppose now that $r$ is the right adjoint to $\overline{f}$ in \cC\ and that the counit of the adjunction is $1_{1_{B/f}}$, so that $\overline{f}r=1_{B/f}$. We have to show that any 2-cell
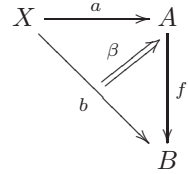
\[
\xymatrix@!=7ex{
X \ar[r]^a \ar[dr]_b^{}="1" & A \ar[d]^f_(.1){}="2" \ar@{=>}"1";"2"^(.4){\beta} \\
& B
}
\]
admits an $f$-cartesian lifting. By the universal property of $B/f$, $\beta$ induces a unique morphism $\widehat{\beta}$
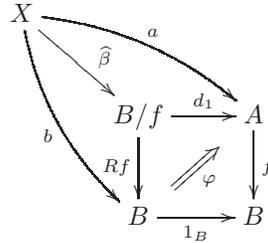
\[
\xymatrix{
X \ar@/_2ex/[ddr]_{b} \ar@/^2ex/[drr]^{a} \ar[dr]^(.6){\widehat{\beta}} \\
& B/f \ar[d]_{Rf} \ar[r]^-{d_1} & A \ar[d]^{f} \\
& B \ar[r]_{1_B} \ar@{}[ur]|(.3){}="1" \ar@{}[ur]|(.7){}="2" \ar@{=>}"1";"2"_{\varphi} & B
}
\]
satisfying the equations
\[
\left\{
\begin{array}{l}
(Rf)\widehat{\beta}=b \\
d_1\widehat{\beta}=a \\
\varphi\widehat{\beta}=\beta. \\
\end{array}
\right.
\]
Let $\alpha$ be equal to $\mu_Ar\widehat{\beta}\colon d_0r\widehat{\beta}\to d_1r\widehat{\beta}=a$. For brevity, we call $a'=d_0r\widehat{\beta}$, the domain of $\alpha$. We are going to show that $\alpha$ is the desired $f$-cartesian lifting of $\beta$. First of all, $f\alpha=f\mu_Ar\widehat{\beta}=\varphi\overline{f}r\widehat{\beta}=\varphi\widehat{\beta}=\beta$. It remains to show that it is cartesian. Consider a morphism $g\colon Y \to X$ and 2-cells $\sigma\colon y \to ag$ and $\gamma\colon fy\to bg$ such that $f\sigma=\beta g\cdot \gamma$.
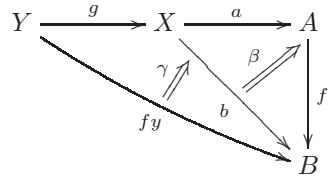
\[
\xymatrix@!=7ex{
Y \ar[r]^g \ar@/_1ex/[drr]_{fy}^{}="3" & X \ar[r]^a \ar[dr]_b^{}="1"_(.2){}="4" \ar@{=>}"3";"4"^{\gamma} & A \ar[d]^f_(.1){}="2" \ar@{=>}"1";"2"^(.4){\beta} \\
& & B
}
\]
The 2-cell $\sigma$ induces a unique morphism $\widehat{\sigma}$
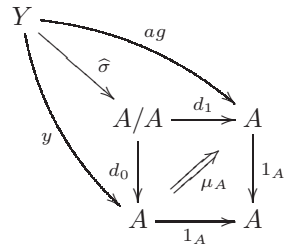
\[
\xymatrix{
Y \ar@/_2ex/[ddr]_{y} \ar@/^2ex/[drr]^{ag} \ar[dr]^(.6){\widehat{\sigma}} \\
& A/A \ar[d]_{d_0} \ar[r]^-{d_1} & A \ar[d]^{1_A} \\
& A \ar[r]_{1_A} \ar@{}[ur]|(.3){}="1" \ar@{}[ur]|(.7){}="2" \ar@{=>}"1";"2"_{\mu_A} & A
}
\]
satisfying the equations
\[
\left\{
\begin{array}{l}
d_0\widehat{\sigma}=y \\
d_1\widehat{\sigma}=ag \\
\mu_A\widehat{\sigma}=\sigma. \\
\end{array}
\right.
\]
Consider the diagram
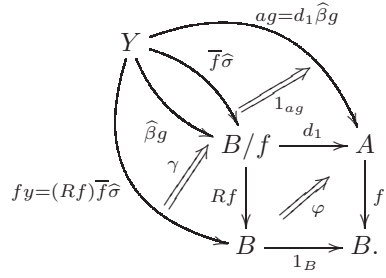
\[
\xymatrix{
Y \ar@/_6ex/[ddr]_{fy=(Rf)\overline{f}\widehat{\sigma}}^(.7){}="3" \ar@/^6ex/[drr]^{ag=d_1\widehat{\beta}g}_(.7){}="2" \ar@/_2ex/[dr]_{\widehat{\beta}g}_(.8){}="4" \ar@/^2ex/[dr]^{\overline{f}\widehat{\sigma}}^(.8){}="1" \ar@{=>}"1";"2"_{1_{ag}} \ar@{=>}"3";"4"^{\gamma} \\
& B/f \ar[d]_{Rf} \ar[r]^-{d_1} & A \ar[d]^{f} \\
& B \ar[r]_{1_B} \ar@{}[ur]|(.3){}="1" \ar@{}[ur]|(.7){}="2" \ar@{=>}"1";"2"_{\varphi} & B.
}
\]
Since
\[
\left\{
\begin{array}{l}
f1_{ag}\cdot\varphi\overline{f}\widehat{\sigma}=\varphi\overline{f}\widehat{\sigma}=f\mu_A\widehat{\sigma}=f\sigma=\beta g\cdot\gamma \\
\varphi\widehat{\beta}g\cdot\gamma=\beta g\cdot\gamma \\
\end{array}
\right.
\]
by the 2-dimensional universal property of $B/f$, there exists a unique 2-cell $\theta\colon \overline{f}\widehat{\sigma}\to \widehat{\beta}g$ such that
\[
\left\{
\begin{array}{l}
(Rf)\theta=\gamma \\
d_1\theta=1_{ag}. \\
\end{array}
\right.
\]

Considering the unit $\eta\colon 1_{A/A}\to r\overline{f}$ of the adjunction $\overline{f}\dashv r$, let $\tau$ be equal to $d_0(r\theta\cdot\eta\widehat{\sigma})\colon y \to a'g$. We want to show that $\tau$ is the unique 2-cell satisfying the equations
\begin{equation} \label{tau}
\left\{
\begin{array}{l}
\alpha g\cdot \tau=\sigma \\
f\tau=\gamma. \\
\end{array}
\right.
\end{equation}
Let us start with the second equation:
\[
f\tau=fd_0(r\theta\cdot\eta\widehat{\sigma})=(Rf)\overline{f}(r\theta\cdot\eta\widehat{\sigma})=(Rf)(\theta\cdot\overline{f}\eta\widehat{\sigma})=(Rf)(\theta\cdot 1_{\overline{f}\widehat{\sigma}})=\gamma,
\]
where the last but one equality follows from the fact that $\overline{f}\eta$ is an identity since the counit is, by the triangle equalities. To prove the first equation we will use interchange law following from horizontal composition of 2-cells in the following diagram:
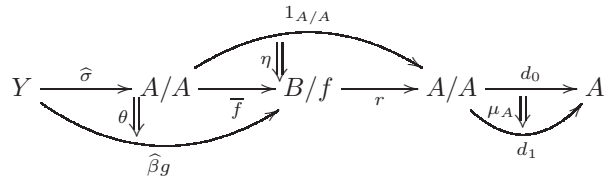
\[
\xymatrix@!=7ex{
Y \ar[r]^-{\widehat{\sigma}}_(.8){}="5" \ar@/_5ex/[rr]_-{\widehat{\beta}g}_(.4){}="6" \ar@{=>}"5";"6"_{\theta} & A/A \ar@/^5ex/[rr]^-{1_{A/A}}_(.4){}="1" \ar[r]_-{\overline{f}}^(.8){}="2" \ar@{=>}"1";"2"_{\eta} & B/f \ar[r]_-{r} & A/A \ar[r]^-{d_0}_{}="3" \ar@/_4ex/[r]_-{d_1}^{}="4" \ar@{=>}"3";"4"_{\mu_A} & A
}
\]
\begin{align*}
\alpha g\cdot\tau & =\alpha g\cdot d_0(r\theta\cdot\eta\widehat{\sigma})=\mu_A r\widehat{\beta}g\cdot d_0r\theta\cdot d_0\eta\widehat{\sigma} \\
& = d_1r\theta\cdot\mu_A r\overline{f}\widehat{\sigma}\cdot d_0\eta\widehat{\sigma}=d_1\theta\cdot(\mu_Ar\overline{f}\cdot d_0\eta)\widehat{\sigma} \\
& = 1_{ag}\cdot(d_1\eta\cdot\mu_A)\widehat{\sigma}=d_1\overline{f}\eta\widehat{\sigma
}\cdot\mu_A\widehat{\sigma}=1_{ag}\cdot\sigma=\sigma\,.
\end{align*}

It just remains to prove that $\tau$ is unique. Let us consider the following diagram, where $\widehat{1_y}$ is the 1-cell induced from $1_y$ by the universal property of $A/A$:
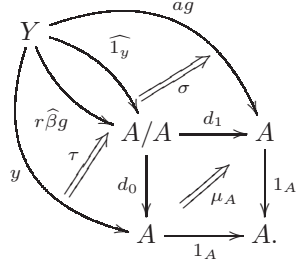
\[
\xymatrix{
Y \ar@/_6ex/[ddr]_{y}^(.7){}="3" \ar@/^6ex/[drr]^{ag}_(.7){}="2" \ar@/_2ex/[dr]_{r\widehat{\beta}g}_(.8){}="4" \ar@/^2ex/[dr]^{\widehat{1_{y}}}^(.8){}="1" \ar@{=>}"1";"2"_{\sigma} \ar@{=>}"3";"4"^{\tau} \\
& A/A \ar[d]_{d_0} \ar[r]^-{d_1} & A \ar[d]^{1_A} \\
& A \ar[r]_{1_A} \ar@{}[ur]|(.3){}="1" \ar@{}[ur]|(.7){}="2" \ar@{=>}"1";"2"_{\mu_A} & A.
}
\]
Since
\[
\left\{
\begin{array}{l}
\sigma\cdot\mu_A\widehat{1_y}=\sigma\cdot 1_y=\sigma \\
\mu_Ar\widehat{\beta}g\cdot\tau=\alpha g\cdot\tau=\sigma, \\
\end{array}
\right.
\]
by the 2-dimensional universal property of $A/A$, there exists a unique 2-cell $\xi\colon \widehat{1_y}\to r\widehat{\beta}g$ such that
\[
\left\{
\begin{array}{l}
d_0\xi=\tau \\
d_1\xi=\sigma. \\
\end{array}
\right.
\]
For any other 2-cell $\tau'\colon y\to a'g$ satisfying (\ref{tau}), one can define $\xi'$ as for $\xi$ above. Now $\tau'=\tau$ will follow from $\xi'=\xi$. We actually show the equality of their corresponding 2-cells in the adjunction, which are the two composites
\[
\xymatrix{
\overline{f}\widehat{1_y} \ar@<.7ex>[r]^-{\overline{f}\xi} \ar@<-.7ex>[r]_-{\overline{f}\xi'} & \overline{f}r\widehat{\beta}g \ar[r]^{\epsilon\widehat{\beta}g} & \widehat{\beta}g.
}
\]
Now recalling that the counit $\epsilon$ is an identity, it suffices to prove that $\overline{f}\xi'=\overline{f}\xi$. But the latter follows by uniqueness since both fill the left upper part of the diagram below and have the same projections through $Rf$ and $d_1$.
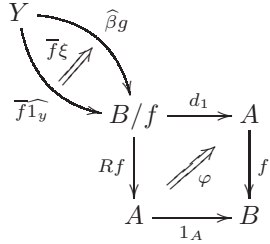
\[
\xymatrix{
Y \ar@/_3ex/[dr]_{\overline{f}\widehat{1_y}}^{}="1" \ar@/^3ex/[dr]^{\widehat{\beta}g}_{}="2" \ar@{=>}"1";"2"^{\overline{f}\xi} \\
& B/f \ar[d]_{Rf} \ar[r]^-{d_1} & A \ar[d]^{f} \\
& A \ar[r]_{1_A} \ar@{}[ur]|(.3){}="1" \ar@{}[ur]|(.7){}="2" \ar@{=>}"1";"2"_{\varphi} & B
}
\]
\[
\left\{
\begin{array}{l}
(Rf)\overline{f}\xi=fd_0\xi=f\tau=\gamma=f\tau'=fd_0\xi'=(Rf)\overline{f}\xi' \\
d_1\overline{f}\xi=d_1\xi=\sigma=d_1\xi'=d_1\overline{f}\xi' \\
\end{array}
\right.
\]

\end{proof}

The following is the analogous of Chevalley criterion for opfibrations.

\begin{Proposition}
A 1-cell $f$ in \cC\ is an opfibration if and only if the 1-cell $\overline{f}$ below
\[
\xymatrix{
A/A \ar@/_2ex/[ddr]_{d_0} \ar@/^2ex/[drr]^{fd_1} \ar[dr]^(.6){\overline{f}} \\
& f/B \ar[d]_{p_0} \ar[r]^-{Lf} & B \ar[d]^{1_B} \\
& A \ar[r]_{f} \ar@{}[ur]|(.3){}="1" \ar@{}[ur]|(.7){}="2" \ar@{=>}"1";"2"_{\psi} & B
}
\]
uniquely determined by the equations
\[
\left\{
\begin{array}{l}
p_0\overline{f}=d_0 \\
(Lf)\overline{f}=fd_1 \\
\psi\overline{f}=f\mu_A \\
\end{array}
\right.
\]
has a left adjoint with unit an identity.
\end{Proposition}

\end{document}